\documentclass[11pt]{amsart}

\usepackage[margin=1.05in]{geometry}
\usepackage{amsmath,amssymb,amsthm,mathtools}
\usepackage{microtype}
\usepackage{hyperref}
\usepackage[nameinlink,capitalise]{cleveref}

\numberwithin{equation}{section}

\hypersetup{colorlinks=true,linkcolor=blue,citecolor=blue,urlcolor=blue}

\newtheorem{theorem}{Theorem}[section]
\newtheorem{proposition}[theorem]{Proposition}
\newtheorem{lemma}[theorem]{Lemma}
\theoremstyle{definition}

\newtheorem{corollary}[theorem]{Corollary}

\DeclareMathOperator{\tr}{tr}
\DeclareMathOperator{\osc}{osc}
\DeclareMathOperator{\Lip}{Lip}
\DeclareMathOperator{\diver}{div}
\newcommand{\R}{\mathbb R}
\newcommand{\Tc}{\Theta_c}
\newcommand{\ds}{\delta_*}
\newcommand{\dd}{\,dx}
\title[Gradient estimates with Lipschitz critical phase]{Gradient estimates for the Lagrangian mean curvature equation with Lipschitz critical phase}

\author{Ruosi Chen, Xingchen Zhou and Ruixuan Zhu}

\address{Ruosi Chen, 
Department of Mathematical Sciences, Tsinghua University, Beijing 100084, China}
\email{crs22@mails.tsinghua.edu.cn}

\address{Xingchen Zhou, 
School of Mathematics and Statistics, Hainan University, Haikou, 570228, PR China.}
\email{zxc3zxc4zxc5@stu.xjtu.edu.cn}

\address{Ruixuan Zhu, 
Institute for Theoretical Sciences, Westlake University, Hangzhou, 310030, China.}
\email{zhuruixuan@westlake.edu.cn}

\date{}

\begin{document}
\begin{abstract}
Let $n\ge3$. We prove a priori interior gradient estimate for the Lagrangian mean curvature equation
\[
 \sum_{i=1}^n\arctan\lambda_i(D^2u)=\theta(x)
\]
under the critical and supercritical phase condition
$\theta\ge (n-2)\pi/2$, assuming only that the Lipschitz norm of the prescribed phase is bounded. 
\end{abstract}
\maketitle

\section{Introduction}
We consider the Lagrangian mean curvature equation
\begin{equation}\label{eq:phase}
\sum_{i=1}^n\arctan\lambda_i(D^2u)=\theta(x)
\end{equation}
in a domain in $\R^n$, where $\lambda_i(D^2u)$ are the eigenvalues of the Hessian matrix. We write $B_r=B_r(0)$ and denote the positive critical phase by $\Tc=(n-2)\pi/2$. We seek an interior gradient bound in terms of the oscillation of $u$ and the Lipschitz seminorm of $\theta$, allowing $\theta$ to attain $\Tc$.

When $\theta$ is constant, \eqref{eq:phase} is the special Lagrangian equation. Its level sets are convex at critical and supercritical phases \cite[Lemma 2.1]{Yuan}. Wang--Yuan \cite{WY} proved interior Hessian estimates in this range for all $n\ge3$. Shankar \cite{Shankar} obtained another proof by a doubling argument.

For prescribed phase $\theta(x)$, Bhattacharya--Mooney--Shankar \cite{BMS} proved interior gradient estimates when $\theta\in C^2$ and $\theta\ge\Tc$. Their proof also applies to certain $C^1$ phases satisfying a first order differential inequality. Bhattacharya--Shankar--Wall--Yepez \cite[Theorem 1.2]{BSWY} considered phases depending also on $u$ and $Du$. Their assumptions include uniform $C^2$ bounds in $x$ and $|D_x\theta|\le C(\theta-\Tc)^{1/2}$. These assumptions do not cover a general Lipschitz prescribed phase attaining $\Tc$. We establish an interior gradient estimate with no additional condition on $D\theta$ near the critical phase.

\begin{theorem}\label{thm:main}
Let $n\ge3$ and let $u$ be a smooth solution of \eqref{eq:phase} in $B_4$. Assume that $\theta$ is Lipschitz and
\[
\Tc\le\theta(x)<\frac{n\pi}{2}\quad\text{in }B_4.
\]
If $M:=\osc_{B_4}u<\infty$ and $K:=\Lip_{B_4}\theta<\infty$, then
\begin{equation}\label{eq:main}
\sup_{B_1}|Du|\le C(n,M,K).
\end{equation}
The constant is independent of higher derivatives of $\theta$.
\end{theorem}

Combining Theorem~\ref{thm:main} with the Hessian estimate of Ding \cite[Theorem 1.1]{Ding}, we obtain the following consequence.

\begin{corollary}\label{cor:hessian}
Under the assumptions and notation of Theorem~\ref{thm:main},
\[
 \|D^2u\|_{L^{\infty}(B_{1/2})}\le C(n,M,K).
\]
\end{corollary}

These estimates also yield interior regularity for continuous viscosity solutions. We use an approximation argument as in Zhou \cite[Section 3]{Zhou}, solving smooth, strictly supercritical equations with a common boundary value on an interior ball. At the critical endpoint, we identify the limit by testing against small quadratic perturbations of the resulting classical solution. Thus the argument does not require a comparison principle between arbitrary viscosity solutions at the critical phase. The details are given in Section~4.

\begin{corollary}\label{cor:regularity}
Let $n\ge3$ and let $u\in C^0(B_4)$ be a viscosity solution of \eqref{eq:phase}, where $\theta$ is locally Lipschitz and $\Tc\le\theta<n\pi/2$ in $B_4$. Then $u\in C^{2,\alpha}_{\mathrm{loc}}(B_4)$ for every $\alpha\in(0,1)$.
\end{corollary}

We give a brief explanation of our approach. When the phase stays a fixed positive distance above $\Tc$, semiconvexity gives the gradient estimate directly. It therefore suffices to treat phases close to $\Tc$.

The main difficulty is to keep the gradient bound uniform as $\theta$ approaches $\Tc$. Near $\Tc$, we write \eqref{eq:phase} in terms of two polynomials in $D^2u$. Their matrix derivatives have zero divergence. Our key observation is that the slope
$p=(1+|Du|^2)^{1/2}$ satisfies an almost Jacobi inequality that involves only $D\theta$.

To apply Moser iteration on the graph $(x,Du(x))$, we need an initial integral bound for the gradient. By $(n-1)$-convexity and the estimates of Trudinger--Wang \cite{TW}, we control the required integrals against Hessian measures of order at most $n-2$. For the remaining term of order $n-1$, we use the fact that its coefficient vanishes at $\theta=\Tc$. A local semiconvexity argument bounds $(\theta-\Tc)|Du|$. We then use the polynomial equation to control this term by graph volume and lower Hessian measures. With the initial integral bound in hand, we apply the Michael--Simon Sobolev inequality \cite{MS} and iterate to obtain the gradient estimate. The restriction $n\ge3$ is used for the Sobolev exponent $2n/(n-2)$.

The paper is organized as follows. In Section~2, we derive the differential inequality for the gradient. In Section~3, we prove the gradient estimate near the critical phase. In Section~4, we prove Theorem~\ref{thm:main} and Corollary~\ref{cor:regularity}.

\section{An almost Jacobi inequality for the gradient}

We first derive the inequality used to estimate the gradient near the critical phase. Let $u\in C^3(B_4)$ solve \eqref{eq:phase}, and let $\varphi=\theta-\Tc$. Assume
\begin{equation}\label{eq:near}
0\le\varphi\le\ds<\frac{\pi}{2}\quad\text{in }B_4,
\end{equation}
where $\ds>0$ is fixed. We use $M=\osc_{B_4}u<\infty$ and $K=\Lip_{B_4}\theta<\infty$. Let $a=\tan\varphi$ and $p=(1+|Du|^2)^{1/2}$. Since $u\in C^3$, the left side of \eqref{eq:phase} is $C^1$. Thus $\theta$ is $C^1$ and $|D\theta|\le K$. In particular, $|Da|\le\sec^2\ds\,K$. We do not differentiate $\theta$ twice.

For a symmetric matrix $A$, define $R$ and $S$ by
\[
R(A)+iS(A)=i^{2-n}\det(I+iA).
\]
Let $\sigma_j(A)$ be the $j$th elementary symmetric polynomial of the eigenvalues of $A$. We set $\sigma_0=1$ and $\sigma_j=0$ for $j<0$ or $j>n$. Expanding the determinant, we have
\[
\begin{aligned}
R&=\sum_{m\ge0}(-1)^{m+1}\sigma_{n-2m},\\
S&=\sum_{m\ge0}(-1)^m\sigma_{n-1-2m}.
\end{aligned}
\]
All sums are finite under this convention. For $A=D^2u$, let $V=\sqrt{\det(I+A^2)}$. By \eqref{eq:phase}, $R=V\cos\varphi>0$ and
\begin{equation}\label{eq:RS}
S=V\sin\varphi=aR.
\end{equation}
Let $F=S_A-aR_A$ and $Q=S_A$, where the subscripts denote matrix derivatives. Thus $F$ is the linearization of $S-aR$ with $a$ held fixed.

We first show that $F$ is positive definite. In coordinates where $A$ is diagonal, denote the diagonal entries of $F$ and $Q$ by $F_i$ and $Q_i$. By differentiating the determinant with respect to $\lambda_i$, we obtain
\[
\partial_{\lambda_i}(R+iS)
=(R+iS)\frac{\lambda_i+i}{1+\lambda_i^2}.
\]
At $A=D^2u$, we have $F_i=(R+aS)/(1+\lambda_i^2)=V\sec\varphi/(1+\lambda_i^2)$. Thus
\begin{equation}\label{eq:conformal}
F^{ij}=V\sec\varphi\,g^{ij},
\end{equation}
where $g=I+A^2$ is the metric of the graph $(x,Du(x))$, and $(g^{ij})=g^{-1}$. In particular, $F$ is positive definite.

We also need positivity of $Q$. This holds throughout the positive branch $\Tc\le\theta<n\pi/2$. To see this, let $\alpha_i=\arctan\lambda_i$ and $\beta_i=\sum_{j\ne i}\alpha_j-\Tc$. Since $\theta\ge\Tc$ and each $\alpha_j<\pi/2$, we have $-\pi/2<-\alpha_i\le\beta_i<\pi/2$. The determinant derivative therefore implies
\[
Q_i=\frac{V}{\sqrt{1+\lambda_i^2}}\cos\beta_i\in(0,V].
\]
Under \eqref{eq:near}, this is also $Q_i=R(1+a\lambda_i)/(1+\lambda_i^2)$. Consequently,
\begin{equation}\label{eq:volume-trace}
\begin{aligned}
0&<\tr Q\le nV,\\
0&<\tr F\le n\sec\ds\,V.
\end{aligned}
\end{equation}

The divergence calculation requires a second expression for $Q$. Let $T_j=\partial\sigma_{j+1}/\partial A$ be the Newton tensors, with $T_0=I$ and $T_j=0$ for $j<0$ or $j\ge n$. The recursion $T_j=\sigma_jI-AT_{j-1}$ implies
\[
RI-R_AA
=\sum_{m\ge0}(-1)^{m+1}T_{n-2m}
=\sum_{m\ge0}(-1)^mT_{n-2-2m}.
\]
Here the term $T_n$ vanishes. Differentiating the expansion of $S$, we obtain
\begin{equation}\label{eq:Qidentity}
Q=RI-R_AA.
\end{equation}
The tensor expansions are
\[
\begin{aligned}
Q&=\sum_{m\ge0}(-1)^mT_{n-2-2m},\\
F&=\sum_{m\ge0}(-1)^m
\bigl(T_{n-2-2m}+aT_{n-1-2m}\bigr).
\end{aligned}
\]
In dimension three, $R=\sigma_1-\sigma_3$, $S=\sigma_2-1$, and $Q=T_1$.

We record trace bounds for use in the integral estimates. By Wang--Yuan \cite[Lemma 2.1]{WY}, the condition $\theta\ge\Tc$ implies $\sigma_j(D^2u)\ge0$ for $1\le j\le n-1$. In other words, $u$ is $(n-1)$-convex. By the same lemma, $\lambda_{n-1}>0$ and $\lambda_{n-1}\ge|\lambda_n|$ when the eigenvalues are ordered decreasingly. These conclusions apply pointwise to a variable phase. Since $\tr T_j=(n-j)\sigma_j$, we have
\[
\begin{aligned}
\tr Q&=\sum_{m\ge0}2(m+1)(-1)^m\sigma_{n-2-2m},\\
\tr F&=\sum_{m\ge0}(-1)^m
\bigl(2(m+1)\sigma_{n-2-2m}
+a(2m+1)\sigma_{n-1-2m}\bigr).
\end{aligned}
\]
The nonnegativity of the $\sigma_j$ and $a\le\tan\ds$ imply
\begin{equation}\label{eq:lowerQ}
\tr Q\le C_n\sum_{j=0}^{n-2}\sigma_j
\end{equation}
and
\begin{equation}\label{eq:lowerF}
\tr F\le C_{n,\ds}\sum_{j=0}^{n-2}\sigma_j+C_na\sigma_{n-1}.
\end{equation}

We use $D$ and $\diver$ for the Euclidean gradient and divergence. Repeated coordinate indices are summed.

\begin{lemma}\label{lem:differential}
Under \eqref{eq:near}, the function $p=(1+|Du|^2)^{1/2}$ satisfies
\begin{equation}\label{eq:differential}
\diver(F Dp)\ge-C_{\ds}K\tr Q.
\end{equation}
\end{lemma}

\begin{proof}
By differentiating $S(D^2u)-aR(D^2u)=0$, we obtain $F^{ij}u_{ijk}=Ra_k$. The Newton tensors of a Hessian satisfy $\partial_iT_j^{ik}(D^2u)=0$. Hence $S_A$ and $R_A$ are divergence-free, and $\partial_iF^{ij}=-a_iR_A^{ij}$. By \eqref{eq:Qidentity},
\[
\begin{aligned}
\diver(F Du_k)
&=Ra_k-a_iR_A^{ij}u_{jk}\\
&=(RI-R_AD^2u)^{ik}a_i
=Q^{ik}a_i.
\end{aligned}
\]
Using $p_i=u_ku_{ki}/p$, we then have
\[
\diver(F Dp)
=\frac{u_k}{p}\diver(F Du_k)
+\frac{1}{p}F^{ij}u_{ki}u_{kj}
-\frac{1}{p^3}F^{ij}(u_ku_{ki})(u_\ell u_{\ell j}).
\]
Choose a fixed orthogonal coordinate system in which $D^2u$ is diagonal at the point under consideration. Then
\[
\diver(F Dp)
=\frac{QDu\cdot Da}{p}
+\frac{1}{p}\sum_iF_i\lambda_i^2
\left(1-\frac{u_i^2}{p^2}\right).
\]
This calculation uses polynomial matrix derivatives, so repeated eigenvalues cause no difficulty. The sum is nonnegative. Since $Q$ is positive definite,
\[
\frac{|QDu\cdot Da|}{p}
\le |Da|\,\tr Q\,\frac{|Du|}{p}
\le \sec^2\ds\,K\tr Q.
\]
This proves \eqref{eq:differential}.
\end{proof}

\section{The gradient estimate near the critical phase}

We now prove the estimate under \eqref{eq:near}.

\begin{proposition}\label{prop:near}
Fix $0<\ds<\pi/2$. Let $n\ge3$ and let $u\in C^3(B_4)$ solve \eqref{eq:phase}, with $M=\osc_{B_4}u<\infty$ and $K=\Lip_{B_4}\theta<\infty$. If $\Tc\le\theta\le\Tc+\ds$ in $B_4$, then
\[
\sup_{B_1}|Du|\le C(n,M,K,\ds).
\]
\end{proposition}

We retain the notation of Section~2 and let $\gamma=n/(n-2)$. The iteration starts from the following integral bound.

\begin{lemma}\label{lem:initial-integral}
Under the assumptions of Proposition~\ref{prop:near},
\begin{equation}\label{eq:initial-integral}
\int_{B_{3/2}}p^{2\gamma}V\dd\le C(n,M,K,\ds).
\end{equation}
\end{lemma}

\begin{proof}
We first bound the integrals of the Hessian measures. Subtracting $\sup_{B_4}u$, we may assume $-M\le u\le0$. By $(n-1)$-convexity and Trudinger--Wang \cite[Theorem 3.1]{TW}, we obtain
\begin{equation}\label{eq:masses}
\int_{B_3}\sigma_\ell(D^2u)\dd\le C(n,M),
\qquad 0\le\ell\le n-1.
\end{equation}
The case $\ell=0$ is immediate. For the gradient weights, we apply \cite[Theorem 4.1]{TW} with $k=n-1$ to obtain
\[
\int_{B_3}|Du|^q\sigma_\ell(D^2u)\dd
\le C_{n,\ell,q}
\left(\int_{B_4}|u|\dd\right)^{q+\ell},
\]
provided $0\le\ell\le n-2$ and $0\le q<n(n-1-\ell)$. The upper endpoint is at least $n>2$, so $q=1,2$ are admissible. Together with \eqref{eq:masses}, this proves
\begin{equation}\label{eq:weighted}
\int_{B_3}p^q\sigma_\ell(D^2u)\dd\le C(n,M),
\qquad 0\le\ell\le n-2,\quad q=1,2.
\end{equation}

We next bound the graph volume. The graph
\[
\Gamma_u=\{(x,Du(x)):x\in B_4\}\subset\R^{2n}
\]
has metric $g_{ij}=\delta_{ij}+u_{ik}u_{jk}$ and volume form $d\mu=V\dd$. We take $H$ to be the trace of the second fundamental form and let $J(x,y)=(-y,x)$. The mean curvature identity is $H=J\nabla_g\theta$; see \cite[Section 2.1]{Ding}. Thus $|H|^2=g^{ij}\theta_i\theta_j$. For $f\in C_c^1(\Gamma_u)$, the Michael--Simon Sobolev inequality \cite{MS} takes the form
\begin{equation}\label{eq:MS}
\left(\int_{\Gamma_u}|f|^{2\gamma}\,d\mu\right)^{1/\gamma}
\le C_n\int_{\Gamma_u}
\bigl(|\nabla_gf|^2+|H|^2f^2\bigr)\,d\mu.
\end{equation}
To obtain \eqref{eq:MS} from the $L^1$ form, apply it to $|f|^{2(n-1)/(n-2)}$ and use Cauchy's inequality. The graph is $C^2$. A cutoff supported in a compact subset of $B_4$ has compact support on $\Gamma_u$. Its support has finite graph volume because $D^2u$ is continuous.

Choose $0\le\eta\le1$ in $C_c^1(B_3)$, equal to one on $B_{5/2}$, with $|D\eta|\le C$. By \eqref{eq:MS} and $Vg^{ij}=\cos\varphi\,F^{ij}$,
\[
\begin{aligned}
\left(\int_{B_3}\eta^{2\gamma}V\dd\right)^{1/\gamma}
&\le C_n\int_{B_3}
\bigl(Vg^{ij}\eta_i\eta_j+\eta^2Vg^{ij}\theta_i\theta_j\bigr)\dd\\
&\le C_n(1+K^2)\int_{B_3}\tr F\dd
\le C(n,M,K,\ds).
\end{aligned}
\]
For the last inequality we used \eqref{eq:lowerF}, \eqref{eq:masses}, and $a\le\tan\ds$. Therefore
\begin{equation}\label{eq:volume}
\int_{B_{5/2}}V\dd\le C(n,M,K,\ds).
\end{equation}

The term $ap^2\sigma_{n-1}$ is the only part of $p^2\tr F$ not covered by \eqref{eq:weighted}. To bound its integral, we first bound $ap$. We use an elementary semiconvexity estimate that will also enter the proof of Theorem~\ref{thm:main}. Let $0<r\le1$ and $w\in C^2(B_r(x_0))$. If $D^2w\ge-LI$, $L\ge0$, and $\osc_{B_r(x_0)}w\le M$, then
\begin{equation}\label{eq:semiconvex-gradient}
|Dw(x_0)|\le C_n\left(\frac Mr+Lr\right).
\end{equation}
Indeed, $w(x)+L|x-x_0|^2/2$ is convex. Comparing its value at $x_0$ with its values at $x_0\pm(r/2)e$, for each unit vector $e$, proves \eqref{eq:semiconvex-gradient}.

Fix $x\in B_{5/2}$ and let $\delta=\varphi(x)$. When $\delta=0$, we have $a(x)p(x)=0$. For $\delta>0$, choose
\[
\rho=\min\left\{\frac14,\frac{\delta}{2K}\right\},
\]
where the second entry is $+\infty$ if $K=0$. On $B_\rho(x)\subset B_3$, the Lipschitz bound implies $\varphi\ge\delta/2$. For every eigenvalue of $D^2u$ on this ball,
\[
\arctan\lambda_i
=\Tc+\varphi-\sum_{j\ne i}\arctan\lambda_j
>-\frac{\pi}{2}+\frac{\delta}{2}.
\]
Hence $D^2u\ge-\cot(\delta/2)I$ there. By \eqref{eq:semiconvex-gradient},
\[
\delta|Du(x)|
\le C_n\left(\frac{M\delta}{\rho}
+\delta\rho\cot(\delta/2)\right)
\le C(n,M,K,\ds).
\]
Here $\delta/\rho=\max\{4\delta,2K\}\le4\ds+2K$ and $\delta\cot(\delta/2)\le2$. These bounds also cover $K=0$, when $\rho=1/4$. Since $\tan\delta\le C_{\ds}\delta$ and $p\le1+|Du|$, we conclude that
\begin{equation}\label{eq:product}
ap\le C(n,M,K,\ds)\quad\text{in }B_{5/2}.
\end{equation}

We can now prove the weighted trace estimate
\begin{equation}\label{eq:weighted-trace}
\int_{B_{5/2}}p^2\tr F\dd\le C(n,M,K,\ds).
\end{equation}
By \eqref{eq:lowerF} and \eqref{eq:weighted}, only the integral of $ap^2\sigma_{n-1}$ remains to be bounded. The expansion of $S$ and \eqref{eq:RS} imply
\[
\sigma_{n-1}
=aR+\sigma_{n-3}-\sigma_{n-5}+\cdots
\le aR+\sum_{j=0}^{n-3}\sigma_j.
\]
Consequently,
\[
ap^2\sigma_{n-1}
\le (ap)^2R+ap^2\sum_{j=0}^{n-3}\sigma_j.
\]
The integral of the first term is bounded by \eqref{eq:product}, $R\le V$, and \eqref{eq:volume}. The remaining terms are bounded by \eqref{eq:weighted} and $a\le\tan\ds$. This proves \eqref{eq:weighted-trace}.

Choose a cutoff $\eta\in C_c^1(B_{5/2})$, equal to one on $B_2$, with $0\le\eta\le1$ and $|D\eta|\le C$. Testing \eqref{eq:differential} with $\eta^2p$ and applying Cauchy's inequality, we obtain
\[
\int_{B_{5/2}}\eta^2(F Dp)\cdot Dp\dd
\le C_{\ds}K\int_{B_{5/2}}\eta^2p\tr Q\dd
+C\int_{B_{5/2}}p^2(F D\eta)\cdot D\eta\dd.
\]
The first integral on the right is bounded by \eqref{eq:lowerQ} and \eqref{eq:weighted}. By \eqref{eq:weighted-trace}, the cutoff term is also bounded. Thus
\[
\int_{B_2}(F Dp)\cdot Dp\dd\le C(n,M,K,\ds).
\]

Finally, take a cutoff $\eta\in C_c^1(B_2)$, equal to one on $B_{3/2}$, and apply \eqref{eq:MS} to $f=\eta p$. By \eqref{eq:conformal}, the Dirichlet integral is bounded by the preceding energy estimate and \eqref{eq:weighted-trace}. The mean curvature term satisfies
\[
\int_{\Gamma_u}\eta^2p^2|H|^2\,d\mu
=\int_{B_2}\eta^2p^2\cos\varphi\,F^{ij}\theta_i\theta_j\dd
\le K^2\int_{B_2}\eta^2p^2\tr F\dd.
\]
It is therefore bounded by \eqref{eq:weighted-trace}. We obtain
\[
\left(\int_{B_{3/2}}p^{2\gamma}V\dd\right)^{1/\gamma}
\le C(n,M,K,\ds),
\]
which proves \eqref{eq:initial-integral}.
\end{proof}

\begin{proof}[Proof of Proposition~\ref{prop:near}]
Let $1\le r<s\le3/2$ and $q\ge2$. Choose $0\le\eta\le1$ in $C_c^1(B_s)$, equal to one on $B_r$, with $|D\eta|\le C/(s-r)$. Testing \eqref{eq:differential} with $\eta^2p^{q-1}$, we have
\[
\begin{aligned}
(q-1)\int_{B_s}\eta^2p^{q-2}(F Dp)\cdot Dp\dd
&\le C_{\ds}K\int_{B_s}\eta^2p^{q-1}\tr Q\dd\\
&\quad+2\left|\int_{B_s}\eta p^{q-1}
(F Dp)\cdot D\eta\dd\right|.
\end{aligned}
\]
Young's inequality bounds the last term by
\[
\frac{q-1}{2}\int_{B_s}\eta^2p^{q-2}(F Dp)\cdot Dp\dd
+\frac{2}{q-1}\int_{B_s}p^q(F D\eta)\cdot D\eta\dd.
\]
We absorb the first term and use $D(p^{q/2})=(q/2)p^{q/2-1}Dp$. By \eqref{eq:volume-trace} and $p\ge1$, we obtain
\[
\int_{B_s}\eta^2F D(p^{q/2})\cdot D(p^{q/2})\dd
\le C_{n,\ds}\left(qK+\frac{1}{(s-r)^2}\right)
\int_{B_s}p^qV\dd.
\]
We used $q^2/(q-1)\le2q$ and $q^2/(q-1)^2\le4$ to retain the displayed dependence on $q$.

Set $w=p^{q/2}$. By \eqref{eq:conformal}, \eqref{eq:volume-trace}, and the preceding energy estimate,
\[
\begin{aligned}
\int_{\Gamma_u}|\nabla_g(\eta w)|^2\,d\mu
&\le 2\int_{B_s}\eta^2(F Dw)\cdot Dw\dd
+\frac{C_{n,\ds}}{(s-r)^2}\int_{B_s}p^qV\dd\\
&\le C_{n,\ds}\left(qK+\frac{1}{(s-r)^2}\right)
\int_{B_s}p^qV\dd.
\end{aligned}
\]
Apply \eqref{eq:MS} to $\eta w$. By \eqref{eq:conformal} and \eqref{eq:volume-trace}, the mean curvature integral is at most $C_{n,\ds}K^2\int_{B_s}p^qV\dd$. We conclude that
\begin{equation}\label{eq:iteration}
\left(\int_{B_r}p^{\gamma q}V\dd\right)^{1/\gamma}
\le C_{n,\ds}\left(qK+K^2+\frac{1}{(s-r)^2}\right)
\int_{B_s}p^qV\dd.
\end{equation}

Let $q_j=2\gamma^{j+1}$ and $r_j=1+2^{-j-1}$ for $j\ge0$. Apply \eqref{eq:iteration} with $q=q_j$, $r=r_{j+1}$, and $s=r_j$, and take the power $1/q_j$. The resulting multiplicative factors have a bounded product, since
\[
\sum_{j=0}^{\infty}
\frac{\log\!\left[C_{n,\ds}(q_jK+K^2+C4^j)\right]}{q_j}
\le C_{n,\ds}\sum_{j=0}^{\infty}
\frac{j+1+\log(1+K)}{q_j}<\infty.
\]
Iterating, we obtain
\[
\left(\int_{B_1}p^{q_j}V\dd\right)^{1/q_j}
\le C(n,K,\ds)
\left(\int_{B_{3/2}}p^{2\gamma}V\dd\right)^{1/(2\gamma)}
\quad\text{for all }j\ge0.
\]
The right side is bounded by \eqref{eq:initial-integral}. By \eqref{eq:volume}, the measure $V\dd$ is finite on $B_1$. Since $V\ge1$ and $p$ is continuous, we let $j\to\infty$ to obtain
\[
\sup_{B_1}p\le C(n,M,K,\ds).\qedhere
\]
\end{proof}

\section{Proof of the main theorem and interior regularity}

\begin{proof}[Proof of Theorem~\ref{thm:main}]
Fix $\ds\in(0,\pi/8)$ and let $\varphi=\theta-\Tc\ge0$. For $x_0\in B_1$, choose
\[
r_0=\frac{\ds}{32(1+K)}.
\]
Then $r_0<1/32$, $B_{4r_0}(x_0)\subset B_4$, and $Kr_0\le\ds/32$.

If $\varphi(x_0)\ge\ds/4$, then $\varphi\ge7\ds/32>\ds/8$ on $B_{r_0}(x_0)$. Since $\arctan\lambda_i>\varphi-\pi/2$ for every $i$, we have $D^2u\ge-\cot(\ds/8)I$ on this ball. By \eqref{eq:semiconvex-gradient},
\[
|Du(x_0)|
\le C_n\left(\frac{M}{r_0}+r_0\cot(\ds/8)\right)
\le C(n,M,K).
\]

If $\varphi(x_0)<\ds/4$, then on $B_{4r_0}(x_0)$ we have
\[
0\le\varphi
\le\varphi(x_0)+4Kr_0
<\frac{3\ds}{8}<\ds.
\]
Define $v(y)=r_0^{-2}u(x_0+r_0y)$ for $y\in B_4$. Its Hessian is $D^2v(y)=D^2u(x_0+r_0y)$, so $v$ solves \eqref{eq:phase} with prescribed phase $\theta(x_0+r_0y)$. The oscillation of $v$ on $B_4$ is at most $M/r_0^2$, and the Lipschitz seminorm of its phase is at most $r_0K$. Proposition~\ref{prop:near} therefore implies
\[
|Du(x_0)|=r_0|Dv(0)|
\le r_0\,C\left(n,\frac{M}{r_0^2},r_0K,\ds\right)
\le C(n,M,K).
\]
Both cases give a bound uniform in $x_0\in B_1$, proving \eqref{eq:main}.
\end{proof}

\begin{proof}[Proof of Corollary~\ref{cor:regularity}]
Write $\mathcal P(A)=\sum_i\arctan\lambda_i(A)$. We first explain the higher regularity estimates used below. On a bounded set of matrices satisfying $\mathcal P(A)\ge\Tc$, the operator $-e^{-b\mathcal P(A)}$ is concave for a sufficiently large $b$ depending only on the matrix bound and $n\ge3$; see \cite[Lemma 6.1]{BSWY}. Its matrix derivative is uniformly positive definite there. Taking the infimum of its tangent affine functions on the bounded convex set $\{A:-LI\le A\le LI,\ \mathcal P(A)\ge\Tc\}$ gives a concave, uniformly elliptic extension agreeing with the operator on that set. Consequently, uniform interior Hessian bounds and a uniform Lipschitz bound for the phase give uniform interior $C^{2,\beta}$ estimates for some $\beta\in(0,1)$ by Evans--Krylov theory. For smooth solutions $w$ with phase $\vartheta$, differentiating the equation gives
\[
 (I+(D^2w)^2)^{-1}_{ij}(w_k)_{ij}=\vartheta_k.
\]
The coefficients are now uniformly elliptic and uniformly $C^{0,\beta}$. Interior linear $W^{2,q}$ estimates for $w_k$, followed by Sobolev embedding with $q>n/(1-\alpha)$, give uniform interior $C^{2,\alpha}$ estimates for every $\alpha\in(0,1)$. Only the Lipschitz bound for $\vartheta$ is used on the right side.

Fix a ball $\Omega=B_r(x_0)\Subset B_4$. Choose smooth phases $\theta_j$ on $\overline\Omega$ such that
\[
 \Tc<\theta_{j+1}<\theta_j<\frac{n\pi}{2},\qquad
 \theta_j\longrightarrow\theta\ \text{uniformly},\qquad
 \sup_j\Lip_{\Omega}\theta_j<\infty.
\]
By \cite[Corollary 1.1]{BMS}, there are solutions $v_j\in C^\infty(\Omega)\cap C^0(\overline\Omega)$ satisfying
\[
 \mathcal P(D^2v_j)=\theta_j\quad\text{in }\Omega,
 \qquad v_j=u\quad\text{on }\partial\Omega.
\]
Here interior smoothness follows by elliptic bootstrapping from the smoothness of $\theta_j$. The strict inequality $\theta_j>\theta_{j+1}$ and the classical maximum principle give $v_j\le v_{j+1}$. Let $h$ be the harmonic function on $\Omega$ with boundary value $u$. Since the $v_j$ are subharmonic by $(n-1)$-convexity,
\[
 v_1\le v_j\le h\quad\text{on }\overline\Omega.
\]
In particular, their oscillations are uniformly bounded. Theorem~\ref{thm:main} and Corollary~\ref{cor:hessian}, applied on smaller balls after translation and scaling, give uniform interior gradient and Hessian bounds. The preceding regularity estimates and compactness show that the increasing limit $v$ belongs to $C^{2,\alpha}_{\mathrm{loc}}(\Omega)$ for every $\alpha\in(0,1)$ and solves $\mathcal P(D^2v)=\theta$. Since $v_1$ and $h$ are continuous on $\overline\Omega$ and both equal $u$ on $\partial\Omega$, the same bounds show that $v$ extends continuously to $\overline\Omega$ with boundary value $u$.

It remains to identify $v$ with the given viscosity solution. For $t>0$, set
\[
 v_t^-=v+t(|x-x_0|^2-r^2),\qquad
 v_t^+=v-t(|x-x_0|^2-r^2).
\]
Strict ellipticity of $\mathcal P$ gives
\[
 \mathcal P(D^2v_t^-)>\theta,\qquad
 \mathcal P(D^2v_t^+)<\theta\quad\text{in }\Omega.
\]
Both functions equal $u$ on $\partial\Omega$. A positive maximum of $v_t^--u$ would give a $C^2$ test function touching $u$ from below and contradict its viscosity supersolution inequality. Similarly, a positive maximum of $u-v_t^+$ contradicts its viscosity subsolution inequality. Thus $v_t^-\le u\le v_t^+$, and letting $t\downarrow0$ gives $u=v$. Since $\Omega\Subset B_4$ was arbitrary, the corollary follows.
\end{proof}

\bigskip

\noindent \textbf{Declaration on the use of AI: } We acknowledge the assistance of AI in verifying the symmetric polynomial formulae and the gradient Jacobi inequality, and for language polishing of the manuscript.
The authors have reviewed and verified all AI-assisted material and take full responsibility for the content of this paper.

\end{document}